\documentclass[12pt,oneside,english,reqno]{amsart}
\usepackage{charter}
\usepackage[T1]{fontenc}
\usepackage[latin9]{inputenc}
\usepackage{geometry}
\usepackage{babel}
\usepackage{prettyref}
\usepackage{amsthm}
\usepackage{amstext}
\usepackage{amssymb}
\usepackage{setspace}
\usepackage[numbers]{natbib}
\usepackage[unicode=true,
 bookmarks=false,
 breaklinks=false,pdfborder={0 0 1},backref=false,colorlinks=false]
 {hyperref}

\makeatletter
\theoremstyle{plain}
\newtheorem{thm}{\protect\theoremname}
  \theoremstyle{definition}
  \newtheorem{defn}[thm]{\protect\definitionname}
  \theoremstyle{remark}
  \newtheorem{rem}[thm]{\protect\remarkname}
  \theoremstyle{plain}
  \newtheorem{cor}[thm]{\protect\corollaryname}
  \theoremstyle{plain}
  \newtheorem*{conjecture*}{\protect\conjecturename}

\usepackage{eucal}
\let\mathcal=\CMcal
\usepackage{dsfont}

\DeclareMathOperator{\VN}{VN}

\newrefformat{def}{Definition \ref{#1}}
\newrefformat{lem}{Lemma \ref{#1}}
\newrefformat{pro}{Proposition \ref{#1}}
\newrefformat{prop}{Proposition \ref{#1}}
\newrefformat{cor}{Corollary \ref{#1}}
\newrefformat{thm}{Theorem \ref{#1}}
\newrefformat{exa}{Example \ref{#1}}
\newrefformat{rem}{Remark \ref{#1}}
\newrefformat{sec}{\S \ref{#1}}
\newrefformat{conj}{Conjecture \ref{#1}}

\makeatother

  \providecommand{\conjecturename}{Conjecture}
  \providecommand{\corollaryname}{Corollary}
  \providecommand{\definitionname}{Definition}
  \providecommand{\remarkname}{Remark}
\providecommand{\theoremname}{Theorem}

\begin{document}
\global\long\def\e{\varepsilon}
\global\long\def\N{\mathbb{N}}
\global\long\def\Z{\mathbb{Z}}
\global\long\def\Q{\mathbb{Q}}
\global\long\def\R{\mathbb{R}}
\global\long\def\C{\mathbb{C}}
\global\long\def\norm#1{\left\Vert #1\right\Vert }

\global\long\def\H{\EuScript H}
\global\long\def\a{\alpha}
\global\long\def\be{\beta}

\global\long\def\tensor{\otimes}

\global\long\def\A{\forall}

\global\long\def\o{\omega}
\global\long\def\G{\mathbb{G}}
\global\long\def\Linfty#1{L^{\infty}(#1)}
\global\long\def\Lone#1{L^{1}(#1)}
\global\long\def\Ad#1{\mathrm{Ad}(#1)}
\global\long\def\id{\mathrm{id}}
\global\long\def\one{\mathds{1}}

\title{A note on amenability of locally compact quantum groups}

\author{Piotr M.~So{\l}tan}

\address{Department of Mathematical Methods in Physics, Faculty of Physics,
University of Warsaw, Poland}

\email{piotr.soltan@fuw.edu.pl}

\author{Ami Viselter}

\address{Department of Mathematical and Statistical Sciences, University of
Alberta, Edmonton, Alberta T6G 2G1, Canada}

\email{viselter@ualberta.ca}
\thanks{The first author was partially supported by National Science Center (NCN) grant
no.~2011/01/B/ST1/05011. The second author was supported by NSERC Discovery Grant no.~418143-2012.}
\begin{abstract}
In this short note we introduce a notion called ``quantum injectivity''
of locally compact quantum groups, and prove that it is equivalent
to amenability of the dual. Particularly, this provides a new characterization
of amenability of locally compact groups.
\end{abstract}
\maketitle

\section*{Introduction}

In abstract harmonic analysis, the connection between amenability
of a locally compact group $G$ and injectivity of its group von Neumann
algebra $\VN(G)$ is well known: while the former always implies the
latter, the converse is generally not true (cf.~Connes \citep[Corollary 7]{Connes__classification_of_inj_factors});
it is, however, true for all discrete groups (and more generally,
all inner-amenable groups). See \citep{Paterson__book_amenability,Runde__book_amenability}
for full details. The notions of amenability and injectivity being
of fundamental importance, it is natural to ask whether the relations
between them carry to the framework of locally compact \emph{quantum}
groups (in the sense of Kustermans and Vaes). Partial answers have
been known for some time. Let $\G$ be a locally compact quantum group.
If $\G$ is amenable then $\Linfty{\hat{\G}}$ is an injective von
Neumann algebra (see Enock and Schwartz \citep{Enock_Schwartz__amenable_Kac_alg}
for Kac algebras, B{\'e}dos and Tuset \citep{Bedos_Tuset_2003} and
Doplicher, Longo, Roberts and Zsid{\'o} \citep{Doplicher_Longo_Roberts_Zsido__QG_action_nucl}
for the general case). Conversely, Ruan \citep{Ruan__amenability}
proved that if $\G$ is a discrete Kac algebra, then injectivity of
$\Linfty{\hat{\G}}$ entails amenability of $\G$. Nevertheless, it
is still an open question whether this holds for general discrete
quantum groups, not necessarily of Kac type.

While attempting to tackle this problem, we found a quantum analogue
of injectivity of von Neumann algebras, which we call ``quantum injectivity''.
Using the structure theory of completely bounded module maps, we prove
that quantum injectivity is equivalent to amenability of the dual
in \emph{all} locally compact quantum groups (not necessarily discrete
or of Kac type). Particularly, we obtain a new characterization of
amenability of locally compact groups. Whether this technique can
be used to solve the open question mentioned above is still yet to
be seen.

\section{preliminaries}

The theory of locally compact quantum groups is by now well established.
In this short note we only give the necessary definitions and facts,
and refer the reader to Kustermans and Vaes \citep{Kustermans_Vaes__LCQG_C_star,Kustermans_Vaes__LCQG_von_Neumann}
for full details.
\begin{defn}
A \emph{locally compact quantum group} (LCQG) in the von Neumann algebraic
setting is a pair $\G=(M,\Delta)$ such that:
\begin{enumerate}
\item $M$ is a von Neumann algebra
\item $\Delta:M\to M\tensor M$ is a co-multiplication, i.e., a unital normal
$*$-homomorphism which satisfies the co-associativity condition:
\[
(\Delta\tensor\id)\Delta=(\id\tensor\Delta)\Delta
\]

\item There exist normal, semi-finite and faithful weights $\varphi,\psi$,
called the left and right Haar weights, that are left and right invariant
(respectively) in the sense that:

\begin{enumerate}
\item $\varphi((\omega\tensor\id)\Delta(x))=\omega(\one)\varphi(x)$ for
all $\o\in M_{*}^{+}$ and $x\in M^{+}$ such that $\varphi(x)<\infty$
\item $\psi((\id\tensor\omega)\Delta(x))=\omega(\one)\psi(x)$ for all $\o\in M_{*}^{+}$
and $x\in M^{+}$ such that $\psi(x)<\infty$.
\end{enumerate}
\end{enumerate}
Following the standard convention, we use the notations $\Linfty{\G},\Lone{\G}$
and $L^{2}(\G)$ for $M,M_{*}$ and the Hilbert space obtained in
the GNS construction of $(M,\varphi)$, respectively. The canonical
injection $\mathcal{N}_{\varphi}\to L^{2}(\G)$ is denoted by $\Lambda_{\varphi}$.
\end{defn}
Every locally compact quantum group $\G$ admits a dual LCQG, denoted
by $\hat{\G}$. This duality extends the classical one for locally
compact abelian groups, and features a Pontryagin-like theorem. Importantly,
$\Linfty{\hat{\G}}$ can be realized canonically over $L^{2}(\G)$.
We say that $\G$ is \emph{compact} if $\varphi(\one)<\infty$ (see
Woronowicz \citep{Woronowicz__symetries_quantiques} for the original
definition, independent of the Kustermans--Vaes axioms), and \emph{discrete}
if $\hat{\G}$ is compact (cf.~\citep{Runde__charac_compact_discr_QG}).
One of the basic objects is the left regular co-representation: it
is the multiplicative unitary $W\in\Linfty{\G}\tensor\Linfty{\hat{\G}}$
given by
\begin{equation}
W^{*}(\Lambda_{\varphi}(a)\tensor\Lambda_{\varphi}(b))=\Lambda_{\varphi\tensor\varphi}(\Delta(b)(a\tensor\one))\qquad\text{for all }a,b\in\mathcal{N}_{\varphi}.\label{eq:left_reg_corep_def}
\end{equation}
It implements the co-multiplication as follows:
\begin{equation}
\Delta(x)=W^{*}(\one\tensor x)W\qquad\text{for all }x\in\Linfty{\G}.\label{eq:left_reg_corep_impl_co_mult}
\end{equation}
Similarly, there is the right regular co-representation, which is
a unitary $V\in\Linfty{\hat{\G}}'\tensor\Linfty{\G}$ satisfying $\Delta(x)=V(x\tensor\one)V^{*}$
for all $x\in\Linfty{\G}$.
\begin{defn}[\citep{Enock_Schwartz__amenable_Kac_alg,Desmedt_Quaegebeur_Vaes,Bedos_Tuset_2003}]
A LCQG $\G$ is called \emph{amenable} if it admits a left-invariant
mean, that is, a state $m\in\Linfty{\G}^{*}$ with $m((\o\tensor\id)\Delta(x))=\o(\one)m(x)$
for all $x\in\Linfty{\G}$, $\o\in\Lone{\G}$.
\end{defn}

\section{Quantum injectivity}

Our main result is the following.
\begin{thm}
\label{thm:quantum_inj}Let $\G$ be a LCQG. The following conditions
are equivalent:
\begin{enumerate}
\item \label{enu:quantum_inj__1}$\G$ is amenable
\item \label{enu:quantum_inj__2}there is a conditional expectation of $B(L^{2}(\G))$
onto $\Linfty{\hat{\G}}$ that maps $\Linfty{\G}$ to $\C\one$
\item \label{enu:quantum_inj__3}there is a conditional expectation of $B(L^{2}(\G))$
onto $\Linfty{\hat{\G}}$ that maps $\Linfty{\G}$ to $\mathrm{Center}(\Linfty{\hat{\G}})$.
\end{enumerate}
\end{thm}
We have recently found out that after we had discovered \prettyref{thm:quantum_inj},
Crann and Neufang \citep{Crann_Neufang__amn_inj_quest} proved a similar
result (from a different perspective), using essentially the same
methods.
\begin{defn}
\label{def:quantum_injective}Let $\G$ be a LCQG. We say that $\G$
is \emph{quantum injective} if there exists a conditional expectation
of $B(L^{2}(\G))$ onto $\Linfty{\G}$ that maps $\Linfty{\hat{\G}}$
to $\mathrm{Center}(\Linfty{\G})$.

Now \prettyref{thm:quantum_inj} simply says that $\G$ is amenable
$\Longleftrightarrow$ $\hat{\G}$ is quantum injective.\end{defn}
\begin{rem}
\label{rem:quantum_injective}A LCQG $\G$ is quantum injective if
and only if there exists a conditional expectation of $B(L^{2}(\G))$
onto either $\Linfty{\G}$ or $\Linfty{\G}'$ that maps either $\Linfty{\hat{\G}}$
or $\Linfty{\hat{\G}}'$ to $\mathrm{Center}(\Linfty{\G})$ (four
equivalent versions). This is a result of the relations $J\Linfty{\G}J=\Linfty{\G}'$,
$J\Linfty{\hat{\G}}J=\Linfty{\hat{\G}}$, $\hat{J}\Linfty{\hat{\G}}\hat{J}=\Linfty{\hat{\G}}'$
and $\hat{J}\Linfty{\G}\hat{J}=\Linfty{\G}$ (\citep{Kustermans_Vaes__LCQG_von_Neumann}).
\end{rem}
If $R$ is a von Neumann algebra over a Hilbert space $\H$, we let
$\mathcal{CB}_{R}(B(\H))$ stand for all completely bounded, $R$-module
maps over $B(\H)$. This space always contains the \emph{elementary
operators}, namely the ones of the form $x\mapsto\sum_{i=1}^{n}a_{i}'xb_{i}'$,
where the $a_{i}',b_{i}'$ belong to $R'$.

For a LCQG $\G$, the co-multiplication $\Delta$ extends to a normal
homomorphism $\overline{\Delta}:B(L^{2}(\G))\to\Linfty{\G}\tensor B(L^{2}(\G))$
given by $\overline{\Delta}(x):=W^{*}(\one\tensor x)W$ for $x\in B(L^{2}(\G))$.
\begin{thm}
\label{thm:cb_module_maps}Let $\G$ be a LCQG. If $E\in\mathcal{CB}_{\Linfty{\hat{\G}}}(B(L^{2}(\G)))$,
then
\[
(\o\tensor\id)\overline{\Delta}(E(x))=E\bigl((\o\tensor\id)\overline{\Delta}(x)\bigr)\qquad(\A x\in B(L^{2}(\G)),\o\in B(L^{2}(\G))_{*}).
\]

\end{thm}
The theorem would be an easy consequence of \eqref{eq:left_reg_corep_impl_co_mult}
if we knew that $E$ was \emph{normal}, but that is rarely true (cf.~\citep[\S V.2, Ex.~8(b)]{Takesaki__book_vol_1}).
\begin{proof}[Proof of \prettyref{thm:cb_module_maps}]
On account of the assumption that $E\in\mathcal{CB}_{\Linfty{\hat{\G}}}(B(L^{2}(\G)))$
there exists by \citep[Theorem 2.5]{Effros_Kishimoto_module_maps}
a net $\left(E_{i}\right)$ of elementary operators with coefficients
in $\Linfty{\hat{\G}}'$ that converges point--ultraweakly%
\footnote{Note that in $\mathcal{CB}_{\Linfty{\hat{\G}}}(B(L^{2}(\G)))$, $w^{*}$-convergence
implies point--ultraweak convergence by \citep[Lemma 2.4]{Effros_Kishimoto_module_maps}.%
} to $E$. Let $\rho,\omega\in B(L^{2}(\G))_{*}$ be given. Fix an
index $i$. Let $\hat{a}_{j}',\hat{b}_{j}'\in\Linfty{\hat{\G}}'$,
$j=1,\ldots,n$, be such that $E_{i}x=\sum_{j=1}^{n}\hat{a}_{j}'x\hat{b}_{j}'$
for all $x\in B(L^{2}(\G))$. Hence, for all $x\in B(L^{2}(\G))$,
\[
(\o\tensor\rho)\overline{\Delta}(E_{i}x)=(\o\tensor\rho)\bigl(\sum_{j=1}^{n}W^{*}(\one\tensor\hat{a}_{j}'x\hat{b}_{j}')W\bigr).
\]
Since $W\in\Linfty{\G}\tensor\Linfty{\hat{\G}}$, we obtain
\[
\begin{split}(\o\tensor\rho)\overline{\Delta}(E_{i}x) & =(\o\tensor\rho)\bigl(\sum_{j=1}^{n}(\one\tensor\hat{a}_{j}')W^{*}(\one\tensor x)W(\one\tensor\hat{b}_{j}')\bigr)\\
 & =(\o\tensor\rho)\bigl(\sum_{j=1}^{n}(\one\tensor\hat{a}_{j}')\overline{\Delta}(x)(\one\tensor\hat{b}_{j}')\bigr)\\
 & =(\rho\circ E_{i})\bigl((\o\tensor\id)\overline{\Delta}(x)\bigr).
\end{split}
\]
In conclusion,
\begin{equation}
(\o\tensor\rho)\overline{\Delta}(E_{i}x)=(\rho\circ E_{i})\bigl((\o\tensor\id)\overline{\Delta}(x)\bigr)\label{eq:E_i_calculations}
\end{equation}
for all $i$. Since $\overline{\Delta}$ is normal, the limit of the
left-hand side of \eqref{eq:E_i_calculations} with respect to $i$
is $(\o\tensor\rho)\overline{\Delta}(Ex)=\rho\bigl((\o\tensor\id)\overline{\Delta}(Ex)\bigr)$,
and that of the right-hand side is $(\rho\circ E)\bigl((\o\tensor\id)\overline{\Delta}(x)\bigr)$.
The foregoing being true for all $\rho\in B(L^{2}(\G))_{*}$, we are
done.
\end{proof}

\begin{proof}[Proof of \prettyref{thm:quantum_inj}]
\prettyref{enu:quantum_inj__3} $\implies$ \prettyref{enu:quantum_inj__1}.
Let $E$ be a conditional expectation of $B(L^{2}(\G))$ onto $\Linfty{\hat{\G}}$
that maps $\Linfty{\G}$ to $\mathrm{Center}(\Linfty{\hat{\G}})$.
Since $E$ is a completely positive $\Linfty{\hat{\G}}$-bimodule
map, it satisfies the conditions of \prettyref{thm:cb_module_maps}.
Fix a state $\rho\in\Linfty{\hat{\G}}^{*}$, and define $m\in\Linfty{\G}^{*}$
by $m:=\rho\circ E|_{\Linfty{\G}}$. So $m$ is a state since $E$
is a conditional expectation. Moreover, for all $x\in\Linfty{\G}$
and $\o\in\Lone{\G}$ we have from \prettyref{thm:cb_module_maps}:
\[
\begin{split}m\bigl((\o\tensor\id)\Delta(x)\bigr) & =\rho\circ E\bigl((\o\tensor\id)\Delta(x)\bigr)=\rho\bigl[(\o\tensor\id)\left(W^{*}(\one\tensor E(x))W\right)\bigr]\\
 & =\rho[(\o\tensor\id)(\one\tensor E(x))]=\o(\one)\rho(E(x))=\o(\one)m(x)
\end{split}
\]
 (because $W\in\Linfty{\G}\tensor\Linfty{\hat{\G}}$). Thus $m$ is
a left-invariant mean of $\G$, which is therefore amenable.

Evidently \prettyref{enu:quantum_inj__2} $\implies$ \prettyref{enu:quantum_inj__3}.
The implication \prettyref{enu:quantum_inj__1} $\implies$ \prettyref{enu:quantum_inj__2}
was established long ago in, e.g., \citep[Theorem 3.3]{Bedos_Tuset_2003}
(see \prettyref{rem:quantum_injective}), but without indicating that
$E(\Linfty{\G})=\C\one$. For completeness, we sketch the argument.
Suppose that $m$ is a left invariant mean on $\Linfty{\G}$, and
denote by $V$ the right regular co-representation of $\G$. For $x\in B(L^{2}(\G))$,
define $E(x)\in B(L^{2}(\G))$ by
\[
\omega(E(x))=m\bigl((\omega\tensor\id)(V(x\tensor\one)V^{*})\bigr)\qquad(\forall\omega\in B(L^{2}(\G))_{*}).
\]
Then $E:B(L^{2}(\G))\to B(L^{2}(\G))$ is clearly unital and positive.
If $x\in\Linfty{\hat{\G}}$, then as $m$ is a mean and $V\in\Linfty{\hat{\G}}'\tensor\Linfty{\G}$,
we have $\omega(E(x))=m\bigl((\omega\tensor\id)(x\tensor\one)\bigr)=\omega(x)$
for all $\omega\in B(L^{2}(\G))_{*}$, so that $E(x)=x$. We have
to show that the range of $E$ is precisely $\Linfty{\hat{\G}}=\Linfty{\hat{\G}}''=\left\{ y\in B(L^{2}(\G)):V(y\tensor\one)V^{*}=y\tensor\one\right\} $
(cf.~\citep{Kustermans_Vaes__LCQG_von_Neumann}, proof of Proposition
4.2). To this end, let $\omega,\rho\in B(L^{2}(\G))_{*}$ be given,
and define $\gamma\in B(L^{2}(\G))_{*}$ by $\gamma(x):=(\omega\tensor\rho)\left[V(x\tensor\one)V^{*}\right]$.
For $x\in B(L^{2}(\G))$ we obtain
\begin{equation}
(\omega\tensor\rho)\left[V(E(x)\tensor\one)V^{*}\right]=\gamma(E(x))=m\bigl((\gamma\tensor\id)(V(x\tensor\one)V^{*})\bigr).\label{eq:range_of_E_1}
\end{equation}
The identity $V_{12}V_{13}=(\id\tensor\Delta)(V)$ shows that
\begin{equation}
(\gamma\tensor\id)(V(x\tensor\one)V^{*})=(\rho\tensor\id)\Delta\left[(\omega\tensor\id)(V(x\tensor\one)V^{*})\right].\label{eq:range_of_E_2}
\end{equation}
By \eqref{eq:range_of_E_1}, \eqref{eq:range_of_E_2} and left invariance
of $m$ we have $(\omega\tensor\rho)\left[V(E(x)\tensor\one)V^{*}\right]=\rho(\one)\omega(E(x))$.
Since $\o$ and $\rho$ were arbitrary, we conclude that $V(E(x)\tensor\one)V^{*}=E(x)\tensor\one$
and hence $E(x)\in\Linfty{\hat{\G}}$ indeed. Finally, if $x\in\Linfty{\G}$,
then $V(x\tensor\one)V^{*}=\Delta(x)$, and the left invariance of
$m$ implies that $\omega(E(x))=m\bigl((\omega|_{\Linfty{\G}}\tensor\id)\Delta(x)\bigr)=\omega|_{\Linfty{\G}}(\one)m(x)=\omega(m(x)\one)$
for all $\omega\in B(L^{2}(\G))_{*}$, so that $E(x)=m(x)\one$.
\end{proof}
As a result, we have the following new characterization of amenability
of groups.
\begin{cor}
Let $G$ be a locally compact group. Then $G$ is amenable $\Longleftrightarrow$
there is a conditional expectation of $B(L^{2}(G))$ onto $\VN(G)$
mapping $\Linfty G$ to $\mathrm{Center}(\VN(G))$ (or to the scalars).
\end{cor}
When $\G$ is a discrete quantum group of Kac type, Ruan \citep{Ruan__amenability}
proved that injectivity of $\Linfty{\hat{\G}}$ implies amenability
of $\G$ by proving directly that $\hat{\G}$ is co-amenable (in Ruan's
nomenclature: $\G$ is strongly {[}Voiculescu{]} amenable). Using
\prettyref{thm:cb_module_maps}, we give a short, direct proof of
this fact (compare \citep[Theorem 4.9]{Bedos_Murphy_Tuset}), in the
same spirit as the proof of the group case. Recall that the left regular
co-representation $\hat{W}$ of $\hat{\G}$ is equal to $\sigma(W)^{*}$,
where $\sigma$ is the flip map $x\tensor y\mapsto y\tensor x$ on
$B(L^{2}(\G))\tensor B(L^{2}(\G))$.
\begin{cor}
If $\G$ is a discrete quantum group of Kac type and $\Linfty{\hat{\G}}$
is injective, then $\G$ is amenable.\end{cor}
\begin{proof}
Since $\hat{\G}$ is compact and of Kac type, its Haar state $\hat{\varphi}$
is a trace. Let $E$ be a conditional expectation from $B(L^{2}(\G))$
onto $\Linfty{\hat{\G}}$. Define $m:=\hat{\varphi}\circ E|_{\Linfty{\G}}$.
If $x\in\Linfty{\G}$ and $\hat{a},\hat{b}\in\Linfty{\hat{\G}}$,
then for $\o:=\o_{\Lambda_{\hat{\varphi}}(\hat{a}),\Lambda_{\hat{\varphi}}(\hat{b})}$
we have from \prettyref{thm:cb_module_maps} and the dual version
of \eqref{eq:left_reg_corep_def}:
\[
\begin{split}m\left((\o\tensor\id)\Delta(x)\right) & =(\o\tensor\hat{\varphi})\overline{\Delta}(E(x))=(\hat{\varphi}\tensor\o)\bigl(\hat{W}(E(x)\tensor\one)\hat{W}^{*}\bigr)\\
 & =\left\langle \hat{W}(E(x)\tensor\one)\hat{W}^{*}(\Lambda_{\hat{\varphi}}(\one)\tensor\Lambda_{\hat{\varphi}}(\hat{a})),\Lambda_{\hat{\varphi}}(\one)\tensor\Lambda_{\hat{\varphi}}(\hat{b})\right\rangle \\
 & =\left\langle (E(x)\tensor\one)\hat{W}^{*}(\Lambda_{\hat{\varphi}}(\one)\tensor\Lambda_{\hat{\varphi}}(\hat{a})),\hat{W}^{*}(\Lambda_{\hat{\varphi}}(\one)\tensor\Lambda_{\hat{\varphi}}(\hat{b}))\right\rangle \\
 & =(\hat{\varphi}\tensor\hat{\varphi})\bigl(\hat{\Delta}(\hat{b}^{*})(E(x)\tensor\one)\hat{\Delta}(\hat{a})\bigr).
\end{split}
\]
By the traciality and invariance of $\hat{\varphi}$ we deduce that
\[
m\left((\o\tensor\id)\Delta(x)\right)=(\hat{\varphi}\tensor\hat{\varphi})\bigl((E(x)\tensor\one)\hat{\Delta}(\hat{a}\hat{b}^{*})\bigr)=\hat{\varphi}(E(x))\hat{\varphi}(\hat{a}\hat{b}^{*})=\o(\one)m(x).
\]
Therefore $m$ is a left invariant mean, and $\G$ is amenable.
\end{proof}
We conclude with the original open question that has been the motivation
for this note. Since a compact quantum group is of Kac type if and
only if its underlying von Neumann algebra is finite (see \citep[Remark A.2]{Soltan__quantum_Bohr_comp};
also compare Fima \citep[Theorem 8]{Fima__LCQG_factors}), this question
is of interest only when $L^{\infty}(\hat{\G})$ is \emph{infinite}
by the last corollary.
\begin{conjecture*}
If $\G$ is a discrete quantum group such that $L^{\infty}(\hat{\G})$
is injective, then $\G$ is amenable (and hence $\hat{\G}$ is co-amenable
by Tomatsu \citep{Tomatsu__amenable_discrete}).
\end{conjecture*}

\section*{Acknowledgments}

We thank M. Kalantar, P. Kasprzak, N. Ozawa, D. Poulin, V. Runde and
R. Tomatsu for interesting discussions about the subject of this note.
The second author is grateful to the team of the Department of Mathematical
Methods in Physics of the University of Warsaw for their hospitality
during his visit there.

\bibliographystyle{amsplain}
\bibliography{Injectivity}

\end{document}